\documentclass{amsart}
\usepackage{amsthm}
\usepackage{amssymb}

\numberwithin{equation}{section} 
\numberwithin{figure}{section} 
\theoremstyle{plain}

\newtheorem{theorem}{Theorem}[section]

\newtheorem{corollary}[theorem]{Corollary}
\newtheorem{definition}[theorem]{Definition}
\newtheorem{lemma}[theorem]{Lemma}
\newtheorem{notation}[theorem]{Notation}
\newtheorem{problem}[theorem]{Problem}
\newtheorem{proposition}[theorem]{Proposition}
\newtheorem{remark}[theorem]{Remark}

\newcommand{\<}{\langle}
\newcommand{\id}{\operatorname{id}}
\renewcommand{\>}{\rangle}

\newcommand{\T}{\mathbb{T}}
\newcommand{\D}{\mathbb{D}}

\newcommand{\Z}{\mathbb{Z}}
\newcommand{\Lb}{\mathcal{L}}

\newcommand{\LiL}{L^{\infty}(\mathbb{T})_{\mathcal{L}}}

\makeatother
\begin{document}

\title{Composition Operators and Endomorphisms}

\thanks{DC and SS were partially supported by the University of Iowa Department
of Mathematics NSF VIGRE grant DMS-0602242.}

\author{Dennis Courtney}
\author{Paul S. Muhly}
\author{Samuel W. Schmidt}

\address{Department of Mathematics\\
 University of Iowa\\
 Iowa City, IA 52242}

\email{dennis-courtney@uiowa.edu}
\email{paul-muhly@uiowa.edu}
\email{samuel-schmidt@uiowa.edu}
\begin{abstract}
If $b$ is an inner function, then composition with $b$ induces an endomorphism,
$\beta$, of $L^\infty(\T)$ that leaves $H^\infty(\T)$ invariant. We
investigate the structure of the endomorphisms of $B(L^2(\T))$  and $B(H^2(\T))$
that implement $\beta$ through the representations of $L^\infty(\T)$ and
$H^\infty(\T)$ in terms of multiplication operators on  $L^2(\T)$  and
$H^2(\T)$. Our analysis, which is based on work of R. Rochberg and J. McDonald,
will wind its way through the theory of composition
operators on spaces of analytic functions to recent work on Cuntz families of
isometries and Hilbert $C^*$-modules.
\end{abstract}
\maketitle

\section{Introduction}

Our objective in this note is to link the venerable theory of composition
operators on spaces of analytic functions to the representation theory
of $C^{*}$-algebras. The theory of composition operators is full
of equations that involve operators that intertwine various types
of representations. In certain situations the equations can be recast
in terms of {}``covariance equations'' that are familiar from the
theory of $C^{*}$-algebras, their endomorphisms and their representations; doing this yields both new theorems and new understanding of known results.

We are inspired in particular by papers by Richard Rochberg \cite{rR73}
and John McDonald \cite{jMcD03}. In \cite[Theorem 1]{rR73}, Rochberg
performs calculations which may be seen from a more contemporary
perspective as identifying certain Cuntz families of isometries and
Hilbert $C^{*}$-modules at the heart of what he is studying. In
\cite{jMcD03}, McDonald built upon Rochberg's work and proved, among
other things, that the canonical transfer operator associated to composition
with a finite Blaschke product leaves the Hardy space $H^{2}(\T)$
invariant. This note is in large part the result of trying to recast
\cite[Theorem 1]{rR73} in the setting of $C^{*}$-algebras and endomorphisms
using McDonald's observation on transfer operators \cite[Lemma 2]{jMcD03}.

The classical Lebesgue and Hardy spaces on the unit circle $\mathbb{T}$
will be denoted by $L^{p}(\mathbb{T})$ and $H^{p}(\mathbb{T})$
respectively. Normalized Lebesgue measure on $\T$ will be denoted $m$. The orthogonal projection from $L^{2}(\mathbb{T})$
onto $H^{2}(\mathbb{T})$ will be denoted by $P$. The usual exponential
orthonormal basis for $L^{2}(\mathbb{T})$ will be denoted by $\{e_{n}\}_{n\in\mathbb{Z}}$,
i.e., $e_{n}(z):=z^{n}$.  We write $(\cdot, \cdot)$ for the inner product of $L^2(\mathbb{T})$. The multiplication operator on $L^{2}(\mathbb{T})$
determined by a function $\varphi\in L^{\infty}(\mathbb{T})$ will
be denoted $\pi(\varphi)$ and the Toeplitz operator on $H^{2}(\mathbb{T})$
determined by $\varphi$ will be denoted by $\tau(\varphi)$, i.e.,
$\tau(\varphi)$ is the restriction of $P\pi(\varphi)P$ to $H^{2}(\mathbb{T})$.

Our use of the notation $\pi$ and $\tau$ is nonstandard. More commonly,
one writes $M_{f}$ for the multiplication operator determined by
$f$ and $T_{f}$ for the Toeplitz operator determined
by $f$, but for the purposes of this note, we have
found the standard notation to be a bit awkward.  In any case, the map $\pi$ is a $C^{*}$-representation of $L^{\infty}(\mathbb{T})$ on $L^{2}(\mathbb{T})$ that is continuous with respect to the weak-$*$
topology on $L^{\infty}(\mathbb{T})$ and the weak operator topology
on $B(L^{2}(\mathbb{T}))$, and $\tau$ is a (completely) positive
linear map from $L^{\infty}(\mathbb{T})$ to $B(H^{2}(\mathbb{T}))$
with similar continuity properties.

We fix throughout an inner function $b$ which at times will further
be assumed to be a finite Blaschke product. Composition with
$b$, that is, the map $\varphi \mapsto \varphi \circ b$, is known to induce a $*$-endomorphism $\beta$ of $L^{\infty}(\mathbb{T})$ that is continuous with respect to the weak-$*$ topology on $L^{\infty}(\T)$.
When $b$ is a finite Blaschke product this statement is fairly elementary; if $b$ is an arbitrary
inner function, it is somewhat more substantial.  We give an operator-theoretic proof in Corollary \ref{Cor:Well defined beta}. When $\beta$ leaves a subspace of $L^{\infty}(\mathbb{T})$ invariant, we will continue
to use the notation $\beta$ for its restriction to the subspace.  The central focus of our analysis is \begin{problem}\label{Problem: Central problem}
Describe all $*$-endomorphisms $\alpha$ of $B(L^{2}(\mathbb{T}))$
such that \begin{equation}
\alpha\circ\pi=\pi\circ\beta\label{eq:cov1}\end{equation}
 and describe all $*$-endomorphisms $\alpha_{+}$ of $B(H^{2}(\mathbb{T}))$
such that \begin{equation}
\alpha_{+}\circ\tau=\tau\circ\beta. \label{eq:cov2}\end{equation}
 \end{problem}

If an endomorphism $\alpha$ of $B(L^{2}(\T))$ satisfies
\eqref{eq:cov1}, the pair $(\pi,\alpha)$ is called a \emph{covariant
representation} of the pair $(L^{\infty}(\T),\beta)$.  As
$\pi$ will be fixed throughout this note, the first part of our problem
is thus to identify all endomorphisms $\alpha$ of $B(L^{2}(\T))$ that
yield a covariant representation $(\pi,\alpha)$ of $(L^{\infty}(\T),\beta)$.
Equation~\eqref{eq:cov2} is a hybrid version of \eqref{eq:cov1},
but as we shall see, it may be interpreted as describing certain covariant
representations of the Toeplitz algebra, i.e., of the $C^{*}$-algebra
$\mathfrak{T}$ generated by all the Toeplitz operators $\tau(\varphi)$,
$\varphi\in L^{\infty}(\T)$. It is not clear \emph{a priori} that
\emph{any} endomorphisms satisfying \eqref{eq:cov1} or \eqref{eq:cov2}
exist. They do, however, as we shall show in Theorem \ref{Thm: Main1},
where Rochberg's work plays a central role. Then, in Corollary
\ref{Thm: Solution2}, we show how Rochberg's analysis yields a complete description of all solutions to \eqref{eq:cov2}. Identifying all solutions to \eqref{eq:cov1} is more complicated, and it is here that we must assume
that $b$ is a finite Blaschke product.  The set of solutions to \eqref{eq:cov1} is described in Theorem~\ref{Thm:Solution 1} under this restriction. 

In solving Problem~\ref{Problem: Central problem}
we obtain many new proofs of known results. We do not take any position
on the matter of which proofs are simpler or more elementary. Our
more modest goal is  to separate what can be derived through elementary
Hilbert space considerations from what requires more specific function-theoretic
analysis. In this respect, we were inspired by the work of Helson
and Lowdenslager \cite{HL61}, Halmos and others who cast Hardy space
theory in Hilbert space terms and, in particular, showed that Beurling's
theorem about invariant subspaces of the shift operator can be proved
with elementary Hilbert space methods.  Indeed, as we shall see, our
main Theorem \ref{Thm: Main1} is a straightforward corollary
of Beurling's theorem and requires no more technology than Helson
and Lowdenslager's approach to that result.  This paper, therefore,
has something of a didactic component. When we
reprove or reinterpret a known result, we call attention to it and
give references to alternative approaches.

\section{Preliminaries and Background}\label{sec:preliminaries}

It is well known that when $H$ is a Hilbert space, $B(H)$ is the
dual space of the space of trace class operators on $H$. The weak-$*$
topology on $B(H)$ is often called the \emph{ultraweak} topology.
We adopt that terminology here. The ultraweak topology is different
from the weak operator topology, but the two coincide on bounded subsets
of $B(H)$. It follows that our representation $\pi$ is continuous
with respect to the weak-$*$ topology on $L^{\infty}(\T)$ and either
the weak operator topology or the ultraweak topology on $B(L^{2}(\T))$.

As indicated earlier, it is straightforward to see that composition
with a \emph{finite} Blaschke product induces an endomorphism of $L^{\infty}(\T)$.
It is less clear that composition with an arbitrary
inner function has this property.  There
are two reasons for this. The first is that the boundary values of
a general inner function $b$ are only defined on a set $F\subseteq\T$
with $m(\T\backslash F)=0$. The second is that an element of $L^{\infty}(\T)$
is an \emph{equivalence class} of measurable functions containing
a bounded representative, where two functions are equivalent
if and only if they differ on a null set. Thus we want to know that
if we extend $b$ arbitrarily on $\T\backslash F$, mapping to $\T$,
and if $\varphi$ and $\psi$ differ at most on a null set, then so
do $\varphi\circ b$ and $\psi\circ b$. A little reflection reveals
that for this to happen, it is necessary and sufficient that the following
assertion be true: 
\begin{quote}
If $b$ is an inner function whose domain on $\T$ is the measurable
set $F$, then for every null set $E$ of $\T$, $b^{-1}(E)$ is a
null set of $F$. 
\end{quote}
This fact is well known, but exactly who deserves credit for
first proving it is unclear to us. The short note by Kametani and
Ugaheri \cite{KU42} proves it in the case that $b(0)=0$.  This implies the general case, as Lebesgue null
sets of $\T$ are preserved by conformal maps of the disc, and every
inner function $b$ can be written $b=\alpha\circ b_{1}$ with 
$b_{1}$ an inner function fixing the origin and $\alpha$ a conformal
map of the disc.  In Corollary~\ref{Cor:Well defined beta}, we will give a proof of this assertion from the abstract Hilbert space perspective that we are promoting.  We will need the following lemma. To emphasize the distinction between
a measurable function $f$ and its equivalence class modulo the relation
of being equal almost everywhere, we \emph{temporarily} write $[f]$
for the latter.

\begin{lemma}\label{lem: Null set} Let $\theta$ be a Lebesgue measurable
function from $\T$ to $\T$. Suppose $\Theta$ is defined on trigonometric
polynomials $p$ by the formula $\Theta(p)=p\circ\theta$. Then 
\begin{enumerate}
\item\label{nullsetone} $\Theta$ has a unique extension to a $*$-homomorphism from $C(\T)$
into $L^{\infty}(\T)$, and it is given by the formula $\Theta(\varphi)=[\varphi\circ\theta]$,
$\varphi\in C(\T)$. 
\item\label{nullsettwo} If $\Theta$ is continuous with respect to the weak-$*$ topology
of $L^{\infty}(\T)$ restricted to $C(\T)$ and the weak-$*$topology
on $L^{\infty}(\T)$, then for each Lebesgue null set $E$ of $\T$,
$m(\theta^{-1}(E))=0$, and thus $\Theta$ extends uniquely to
a $*$-endomorphism of $L^{\infty}(\T)$ satisfying $\Theta([\varphi])=[\varphi\circ\theta]$ for all
$[\varphi]\in L^{\infty}(\T)$.  The map $\Theta$ is completely determined by $[\theta]$.
\end{enumerate}
\end{lemma}

\begin{proof} For the first assertion it suffices to note that if
$p$ is a trigonometric polynomial, then since $\theta$ is assumed
to map $\T$ to $\T$, $\Vert[p\circ\theta]\Vert_{L^{\infty}}=\sup_{z\in\T}\vert p(\theta(z))\vert\leq\sup_{z\in\T}\vert p(z)\vert=\Vert p\Vert_{C(\T)}$.

For the second assertion, fix the Lebesgue null set $E$, and choose a $G_{\delta}$ set
$E_{0}$ containing $E$ such that $E_{0}\backslash E$ has measure
zero. So, if $\{f_{n}\}_{n\geq0}$ is a sequence in $C(\T)$ such
that $f_{n}\downarrow1_{E_{0}}$%
\footnote{We write the characteristic function (or indicator function) of a
set $E$ as $1_{E}$.%
} pointwise, then $[f_{n}]$ converges to $[1_{E_{0}}]=[1_{E}]$ weak-$*$.
But also, $f_{n}\circ\theta\downarrow1_{E_{0}}\circ\theta=1_{\theta^{-1}(E_{0})}$
pointwise. Therefore, $[f_{n}\circ\theta]$ converges to $[1_{E_{0}}\circ\theta]=[1_{\theta^{-1}(E_{0})}]$
weak-$*$. As $E$ is a null set, so is $E_{0}$, and
the $[f_{n}]$ converge to $0$ weak-$*$. Our hypothesis then implies
that the $[\Theta(f_{n})]=[f_{n}\circ\theta]$ converge to $0$ weak-$*$,
proving that $m(\theta^{-1}(E_{0}))=0$.  As $\theta^{-1}(E)\subseteq\theta^{-1}(E_{0})$ it follows that $\theta^{-1}(E)$ is also a null set, as desired.  The remaining assertions are immediate.
\end{proof}

Because of this lemma, if $b$ is an inner function that may be defined only on a subset $F$ of $\T$ with $m(\T\backslash F)=0$, it does no harm to extend $b$ to all of $\T$ by setting $b(z)=1$ for all $z\in\T\backslash F$.

Next, we want to say a few words about $*$-endomorphisms of $B(H)$,
where $H$ is a separable Hilbert space. Our discussion largely follows
Section 2 of \cite{wA89}. A \emph{Cuntz family} on $H$ is an $N$-tuple
of isometries $\{S_{i}\}_{i=1}^{N}$ on $H$ with mutually orthogonal
ranges that together span $H$; here the number $N$ may be a positive
integer or $\infty$. A Cuntz family $S=\{S_{i}\}_{i=1}^{N}$ on $H$
determines a map $\alpha_{S}:B(H)\to B(H)$ via 
\begin{equation}\label{cuntzinduce}
\alpha_{S}(T)=\sum_{i=1}^{N}S_{i}TS_{i}^{*},\qquad T\in B(H).
\end{equation}
 (If $N=\infty$, this sum is convergent in the strong operator topology.)
The map $\alpha_{S}$ is readily seen to be a $*$-endomorphism of
$B(H)$; multiplicativity is deduced from the fact that a tuple $S=\{S_{i}\}_{i=1}^{N}$
is a Cuntz family if and only if the \emph{Cuntz relations} \begin{equation}
S_{i}^{*}S_{j}=\delta_{ij}I,\qquad1\leq i,j\leq N,\label{eq:Cuntz 1}\end{equation}
 and \begin{equation}
\sum_{i=1}^{N}S_{i}S_{i}^{*}=I\label{eq:Cuntz 2}\end{equation}
 are satisfied. (These relations are named after J. Cuntz, who made
a penetrating analysis of them in \cite{jC77}.)

Significantly, \emph{every} $*$-endomorphism $\alpha$ of $B(H)$, with $H$ separable, is of the form $\alpha_{S}$ for some Cuntz family $S$.  We recall the details.  Fixing a $*$-endomorphism $\alpha$, define $E=\{S\in B(H)\mid ST=\alpha(T)S,\, T\in B(H)\}$.  A short calculation shows that for any $S_{1}$ and $S_{2}$ in $E$
the product $S_{2}^{*}S_{1}$ commutes with all elements of $B(H)$,
and is hence a scalar. We may thus define an inner product $\<\cdot,\cdot\>$
on $E$ by the formula 
$$
\< S_{1},S_{2}\> I=S_{2}^{*}S_{1},\qquad S_{1},S_{2}\in E,
$$
and $E$ with this inner product is a Hilbert space.  It is readily checked that any orthonormal basis $S=\{S_{i}\}_{i=1}^{N}$
for $E$ is a Cuntz family satisfying $\alpha=\alpha_{S}$, so it
is enough to know that $E$ \emph{has} an orthonormal basis - that
is, that $E\neq\{0\}$. This follows from the fact that a $*$-endomorphism
of $B(H)$, when $H$ is separable, is necessarily ultraweakly continuous%
\footnote{This non-trivial fact is discussed in detail in \cite[Section V.5]{mT79}.%
}, and that an ultraweakly continuous unital representation of $B(H)$
is necessarily unitarily equivalent to a multiple of the identity
representation of $B(H)$. That multiple is the dimension of $E$.

The correspondence between endomorphisms and Cuntz families is not
quite one-to-one. However, as Laca observed \cite[Proposition 2.2]{mL93},
if $S=\{S_{i}\}_{i=1}^{N}$ and $\tilde{S}=\{\tilde{S}_{i}\}_{i=1}^{N}$
are two Cuntz families such that $\alpha_{S}=\alpha_{\tilde{S}}$,
then there is a unitary matrix $(u_{ij})$ so that $\widetilde{S}_{i}=\sum_{j}u_{ij}S_{j}$,
and conversely. The reason is that $S$ and $\tilde{S}$ are both orthonormal
bases for the same Hilbert space $E$.  (More concretely, one may just
check that the scalars $u_{ij}=S_{j}^{*}\widetilde{S}_{i}$ have the
desired properties.)

Our goal, then, is to describe the collection of Cuntz families $S=\{S_{i}\}_{i=1}^{N}$
on $L^{2}(\mathbb{T})$ and $R=\{R_{i}\}_{i=1}^{N}$ on $H^{2}(\mathbb{T})$
such that $(\pi,\alpha_{S})$ is a covariant representation of $(L^{\infty}(\T),\beta)$
and $(\tau,\alpha_{R})$ is a covariant representation of $(\mathfrak{T},\beta)$
in the sense of equations~\eqref{eq:cov1} and \eqref{eq:cov2}:
\begin{equation}
\sum_{i=1}^{N}S_{i}\pi(\varphi)S_{i}^{*}=\pi(\beta(\varphi))\label{eq:Cuntz1a}, \qquad \varphi \in L^{\infty}(\mathbb{T}),
\end{equation}
 and \begin{equation}
\sum_{i=1}^{N}R_{i}\tau(\varphi)R_{i}^{*}=\tau(\beta(\varphi))\label{eq:Cuntz2a}, \qquad \varphi \in L^{\infty}(\mathbb{T}).
\end{equation}
Finally, we adopt the following notation for Blaschke products.
If $w$ is a nonzero point of the open unit disc $\mathbb{D}$ then
$b_{w}$ will denote the function \[
b_{w}(z):=\frac{|w|}{w}\frac{w-z}{1-\overline{w}z},\]
and  $b_{0}(z):=z$. If $a_{1},a_{2},\ldots,a_{N}$ is a finite list of
not-necessarily-distinct numbers in $\mathbb{D}$, then we will write
$b=\Pi_{j=1}^{N}b_{a_{j}}$ for the Blaschke product with zeros at
$a_{1},a_{2},\ldots,a_{N}$, i.e., multiplicity will be taken into
account.

\section{Rochberg's Observation}

Our analysis hinges on an observation that we learned from R. Rochberg's paper \cite{rR73}.  A preliminary remark on isometries in abstract Hilbert space is useful.  If $V$ is an isometry on a Hilbert space $H$, and $D$ is the subspace $H \ominus  VH$, it is easy to check that the spaces $D, VD, V^2 D, \dots$ are mutually orthogonal, and that $(\bigoplus_{k \geq 0} V^k D)^{\perp} = \bigcap_{j \geq 0} V^j H$, so that $H = \bigoplus_{k \geq 0} V^k D$ if and only if $V$ is \emph{pure} in the sense that $\bigcap_{j \geq 0} V^j H = \{0\}$.  

If $H = H^2(\T)$ and $V$ is the isometry $\tau(b) = \pi(b)|_{H^2(\T)}$ induced by a nonconstant inner function $b$, it turns out that $V$ is pure, and that in fact $D$ is a complete wandering subspace for the unitary $\pi(b)$ in the sense of \eqref{l2span} below.  This is a minor modification of a point made in \cite[Theorem 1]{rR73}.
\begin{lemma}\label{Rochbergs Lemma} Let $b$ be a nonconstant inner function, and let $\mathcal{D}:=H^{2}(\mathbb{T})\ominus\pi(b)H^{2}(\mathbb{T})$.  Then
\begin{equation}\label{h2span}
H^2(\T) = \bigoplus_{k \geq 0} \pi(b)^k \mathcal{D},
\end{equation}
and 
\begin{equation}\label{l2span}
L^{2}(\mathbb{T})=\bigoplus_{n \in \Z} \pi(b)^{n}\mathcal{D}.
\end{equation}
\end{lemma}

\begin{proof}  As we have just observed, equation \eqref{h2span} follows once we know that the space $K:=\bigcap_{n=0}^{\infty}\pi(b)^{n}H^{2}(\mathbb{T})$ is the zero subspace.  But as $\pi(b)$ commutes with $\pi(z)$, the space $K$ is invariant for the unilateral shift $\tau(z) = \pi(z)|_{H^2(\T)}$.  If $K \neq \{0\}$, by Beurling's theorem there is an inner function $\theta$ with $K = \pi(\theta) H^2(\T)$.  As $\pi(b) K = K$ by definition, we see that $\pi(b) \pi(\theta) H^2(\T) = \pi(\theta) H^2(\T)$, and applying $\pi(\theta^{-1})$ to both sides we conclude that $\pi(b) H^2(\T) = H^2(\T)$.  But $b$ is nonconstant, so by the uniqueness assertion in Beurling's theorem (see \cite[Theorem 3]{hH64}), $\pi(b) H^2(\T)$ is a proper subspace of $H^2(\T)$.  This contradiction shows that $K = \{0\}$, and \eqref{h2span} follows.

Since $\pi(b)$ is a unitary on $L^2(\T)$, it is immediate from \eqref{h2span} that the spaces $\pi(b)^n \mathcal{D}$, $n \in \Z$, are mutually orthogonal.  Letting $L = \bigvee_{k \in \Z} \pi(b)^k \mathcal{D}$, it is clear from \eqref{h2span} that $L = \bigvee_{k \geq 0} \pi(b)^{-k} H^2(\T)$, and thus that $L$ is invariant under $\pi(z)$.  By Helson and Lowdenslager's
generalization of Beurling's theorem (see \cite[Section 1]{HL61} or \cite[Theorem 3]{hH64}), either there is a unimodular $\theta \in L^{\infty}(\T)$ with $L =\pi(\theta)H^{2}(\mathbb{T})$ or there is a measurable $E \subseteq \T$ satisfying $L = \pi(1_{E})L^{2}(\mathbb{T})$.  In the first case, as clearly $\pi(b) L = L$, we conclude that $\pi(\theta) \pi(b) H^2(\T) = \pi(\theta) H^2(\T)$, and applying $\pi(\theta^{-1})$ to both sides we conclude that $\pi(b) H^2(\T) = H^2(\T)$, which contradicts the fact that $b$ is not constant.  Thus there is $E \subseteq \T$ with $L = \pi(1_E) L^2(\T)$, and the fact that $L$ contains $H^2(\T)$ implies $E = \T$, so $L = L^2(\T)$ as desired.  
\end{proof}

\begin{corollary}\label{Cor:Well defined beta} If $b$ is an arbitrary
inner function and if $\beta$ is defined on trigonometric polynomials
$p$ by the formula $\beta(p):=p\circ b$, then $\beta$ extends to
a $*$-endomorphism of $L^{\infty}(\T)$ that is continuous with respect
to the weak-$*$ topology. \end{corollary}

\begin{proof}Lemma \ref{Rochbergs Lemma} implies that $\pi(b)$
is unitarily equivalent to a multiple of the bilateral shift - the
multiple being $\dim(\mathcal{D})$. Thus there is a Hilbert space
isomorphism $W$ from $L^{2}(\T)$ to $L^{2}(\T)\otimes\mathcal{D}$
such that $\pi(b)=W^{-1}(\pi(z)\otimes I_{\mathcal{D}})W$. So, for
every trigonometric polynomial $p$, \[
\pi(\beta(p))=p(\pi(b))=W^{-1}p(\pi(z)\otimes I_{\mathcal{D}})W.\]
 Since $\pi$ is a homeomorphism with respect to the weak-$*$ topology
on $L^{\infty}(\T)$ and the ultraweak topology restricted to the
range of $\pi$, it is evident that $b$ and $\beta$ satisfy the
hypotheses of Lemma \ref{lem: Null set}, and the desired result follows. \end{proof}

Of course, the proof just given recapitulates parts of the well-known
theory of the functional calculus for unitary operators.

\begin{theorem}\label{Thm: Main1} Let $b$ be a non-constant inner
function.
If $\{v_{i}\}_{i=1}^{N}$ is an orthonormal basis for $\mathcal{D}=H^{2}(\mathbb{T})\ominus\pi(b)H^{2}(\mathbb{T})$, then there is a unique Cuntz family $S = \{S_i\}_{i=1}^N$ on $L^2(\T)$ satisfying
\begin{equation}\label{cuntzexist}
S_i(e_n) = v_i b^n, \qquad 1 \leq i \leq N, \quad n \in \Z.
\end{equation}
The endomorphism $\alpha_{S}$ determined by $S$ as in \eqref{cuntzinduce} satisfies $
\alpha_S \circ\pi=\pi\circ\beta$, where $\beta$ is the endomorphism $\varphi \mapsto \varphi \circ b$ of $L^{\infty}(\T)$.

Each $S_{i}$ is reduced by $H^{2}(\mathbb{T})$, and if $R_{i}$
is the restriction of $S_{i}$ to $H^{2}(\mathbb{T})$, then $R = \{R_{i}\}_{i=1}^{N}$
is a Cuntz family on $H^{2}(\mathbb{T})$ with the property that $\alpha_{R}\circ\tau=\tau\circ\beta$.
\end{theorem}
The proof of Lemma~\ref{Rochbergs Lemma} showed that $\mathcal{D} = H^2(\T)
\ominus \pi(b) H^2(\T)$ is nonzero, so it has an orthonormal basis; its
dimension $N$ may be finite or infinite.  It is well known that $N$ is finite if
and only if $b$ is a finite Blaschke product. (See Remark \ref{Canonical
Basis}.)
\begin{proof} Lemma~\ref{Rochbergs Lemma} implies that if $v$ is any unit vector in $\mathcal{D}$, the set $\{v b^n: n \in \Z\}$ is an orthonormal set of vectors in $L^2(\T)$.  It follows that for any $1 \leq i \leq N$, there is a unique isometry $S_i$ on $L^2$ satisfying $S_i(e_n) = v_i b^n$ for all $n \in \Z$.  

Lemma~\ref{Rochbergs Lemma} also implies that if $v$ and $w$ are any orthogonal unit vectors in $\mathcal{D}$, the closed linear spans of $\{v b^n: n \in \Z\}$ and $\{w b^n: n \in \Z\}$ are orthogonal.  It follows that the isometries in the tuple $S = \{S_i\}_{i=1}^N$ just defined have orthogonal ranges.  Let $K$ denote the closed linear span of the ranges of the operators $\{S_i\}_{i=1}^N$.  By construction, for all $n \in \Z$ we have $v_i b^n \in K$ for all $1 \leq i \leq N$, and thus $K \supseteq \pi(b)^n \mathcal{D}$ for all $n \in \Z$.  By Lemma~\ref{Rochbergs Lemma} we conclude that $K = L^2(\T)$ and $S$ is a Cuntz family of isometries.

Viewing each $e_{n}$ as an element of $L^{\infty}(\mathbb{T})$, 
it is evident that \begin{equation}
S_{i}\pi(e_{n})=\pi(b^{n})S_{i}=\pi(\beta(e_{n}))S_{i}, \qquad 1 \leq i \leq N, \quad n \in \Z . \label{eq:PrimitiveCovariant}\end{equation}
Since this equation is linear in the
$e_{n}$, we conclude that $S_{i}\pi(p)=\pi(\beta(p))S_{i}$ for every
$i$ and every trigonometric polynomial $p$. Consequently, 
\begin{equation*}
\pi(\beta(p)) = \pi(\beta(p))\sum_{i=1}^{N}S_{i}S_{i}^{*} 
= \sum_{i=1}^{N}S_{i}\pi(p)S_{i}^{*}
= \alpha_S(\pi(p))
\end{equation*}
is satisfied for every trigonometric polynomial $p$. It follows from Corollary \ref{Cor:Well defined beta}
that equation~\eqref{eq:Cuntz1a} is satisfied for all $\varphi\in L^{\infty}(\mathbb{T})$.

The fact that $H^2(\T)$ is invariant under each $S_i$ is immediate from the definition \eqref{cuntzexist}.  As Lemma~\ref{Rochbergs Lemma} implies that $\{v_i b^n: 1 \leq i \leq N, n < 0\}$ is an orthonormal basis of $H^2(\T)^{\perp}$, it is 
also clear from \eqref{cuntzexist} that $H^2(\T)^{\perp}$ is invariant under each $S_i$, so each $S_i$ is reduced by $H^2(\T)$.  The fact that $R$ is a Cuntz family on $H^{2}(\mathbb{T})$ satisfying $\alpha_{R}\circ\tau=\tau\circ\beta$ is then immediate. \end{proof}

Recall that $\mathfrak{T}$ is the $C^{*}$-algebra generated by all
the Toeplitz operators $\{\tau(\varphi)\mid\varphi\in L^{\infty}(\T)\}$.
We shall write $\mathfrak{T}(C(\T))$ for $C^{*}$-subalgebra generated
by the Toeplitz operators with continuous symbols, i.e., $\mathfrak{T}(C(\T))$
is the $C^{*}$-subalgebra of $B(H^{2}(\mathbb{T}))$ generated by
$\{\tau(\varphi)\mid\varphi\in C(\T)\}$. It is well known that $\mathfrak{T}(C(\mathbb{T}))=\{\tau(\varphi)+k\mid\varphi\in C(\mathbb{T}),\, k\in\mathfrak{K}\}$,
where $\mathfrak{K}$ denotes the algebra of compact operators on
$H^{2}(\mathbb{T})$ \cite[7.11 and 7.12]{rD98}.

\begin{corollary}\label{Cor:Extend to Toeplitz Algebra} If $b$
is an inner function, then the map $\tau(\varphi)\to\tau(\varphi\circ b)$,
$\varphi\in L^{\infty}(\T)$, extends to a $*$-endomorphism of $\mathfrak{T}$
that we will continue to denote by $\beta$. Further, $\beta$ leaves
$\mathfrak{T}(C(\T))$ invariant if and only if $b$ is a finite Blaschke
product. Thus, if $\iota$ denotes the identity representation of
$\mathfrak{T}$ on $H^{2}(\T)$, then any solution $\alpha_{+}$ of
equation~\eqref{eq:cov2} (equivalently, any solution $R:=\{R_{i}\}_{i=1}^{N}$
to equation~\eqref{eq:Cuntz2a}) yields a covariant representation
$(\iota,\alpha_{+})$ of $(\mathfrak{T},\beta)$ and $(\iota,\alpha_{+})$
preserves $(\mathfrak{T}(C(\T)),\beta)$ if and only if $b$ is a
finite Blaschke product. \end{corollary}
As elementary as this result seems to be, we do not know how to prove it without recourse to Theorem~\ref{Thm: Main1}.
\begin{proof}The existence of a solution $\alpha_{+}$ to equation
\eqref{eq:cov2} guarantees that the map $\tau(\varphi)\to\tau(\varphi\circ b)$,
$\varphi\in L^{\infty}(\T)$, extends to a $*$-endomorphism of $\mathfrak{T}$,
because $\alpha_{+}$ is a $C^{*}$-endomorphism of a larger $C^{*}$-algebra,
namely $B(H^{2}(\T))$. Thus Theorem \ref{Thm: Main1} shows that
composition with $b$ extends to $\mathfrak{T}$. If $b$ is a finite
Blaschke product then composition with $b$ leaves $C(\T)$ invariant,
i.e., $\beta$ leaves $C(\T)$ invariant. Since the solution $\alpha_{+}$
to equation~\eqref{eq:cov2} is of the form $\alpha_{R}$ where the
Cuntz family $R$ is \emph{finite}, $\alpha_{+}$ leaves $\mathfrak{K}$
invariant and, therefore, it leaves $\mathfrak{T}(C(\T))$ invariant
when $b$ is a finite Blaschke product. Conversely, if $\beta$ leaves
$\mathfrak{T}(C(\mathbb{T}))$ invariant, then letting $\varphi(z)=z$,
we see that $\tau(b)=\alpha_{+}(\tau(\varphi))=\tau\circ\beta(\varphi)$
must be of the form $\tau(f)+k$, for some compact operator $k$ and
some continuous function $f$. But then $\tau(b-f)=k$, and so, by
\cite[7.15]{rD98}, $b=f$ is continuous, and hence a finite Blaschke product. \end{proof}

Rochberg's analysis and Laca's result \cite[Proposition 2.2]{mL93} together yield the following.

\begin{corollary}\label{Thm: Solution2} A Cuntz family $R=\{R_{i}\}_{i=1}^{N}$
in $B(H^{2}(\T))$ satisfies the equation $\alpha_{R}\circ\tau=\tau\circ\beta$
if and only if there is an orthonormal basis $\{v_{i}\}_{i=1}^{N}$
for $\mathcal{D} = H^{2}(\T)\ominus\pi(b)H^{2}(\T)$ so that the $R_{i}$ may be expressed in terms of it as in Theorem \ref{Thm: Main1}. \end{corollary}

\begin{proof}Theorem \ref{Thm: Main1} asserts that if $R$ is a
Cuntz family in $B(H^{2}(\T))$ of the indicated form, then $\alpha_{R}\circ\tau=\tau\circ\beta$.
For the converse, suppose $R:=\{R_{i}\}_{i=1}^{N}$ is a Cuntz family
in $B(H^{2}(\T))$ so that $\alpha_{R}\circ\tau=\tau\circ\beta$.
Then, as we saw in Corollary \ref{Cor:Extend to Toeplitz Algebra},
$\alpha_{R}$ leaves the Toeplitz algebra $\mathfrak{T}$ invariant.
Choose any orthonormal basis $\{v_{i}\}_{i=1}^{N}$ for $\mathcal{D}$
and let $\widetilde{R}=\{\widetilde{R}_{i}\}_{i=1}^{N}$ be corresponding Cuntz
family on $H^{2}(\T)$ obtained from Theorem \ref{Thm: Main1}.
Then the equation $\alpha_{\widetilde{R}}\circ\tau=\tau\circ\beta$
is also satisfied, by Theorem \ref{Thm: Main1}. It follows that $\alpha_{R}$
and $\alpha_{\widetilde{R}}$ agree on $\mathfrak{T}$. Since $\mathfrak{T}$
is ultraweakly dense in $B(H^{2}(\T))$ (because $\mathfrak{T}$ contains
$\mathfrak{K}$) and since $\alpha_{R}$ and $\alpha_{\widetilde{R}}$
are ultraweakly continuous maps of $B(H^{2}(\T)$), $\alpha_{R}=\alpha_{\widetilde{R}}$
on all of $B(H^{2}(\T))$. Thus by \cite[Proposition 2.2]{mL93},
there is a unitary $N\times N$ \emph{scalar} matrix $\left(u_{ij}\right)$
such that $R_{i}=\sum_{j}u_{ij}\widetilde{R}_{j}$. But $\{(\sum_{j}u_{ij}v_{j})\}_{i=1}^{N}$
is also an orthonormal basis of $\mathcal{D}$,
and so the $R_{i}$'s have the desired form.\end{proof}

\begin{remark}\label{Canonical Basis} It was previously remarked that the
nonconstant inner function $b$
is a finite Blaschke product if and only if the space
$\mathcal{D}=H^{2}(\mathbb{T})\ominus\pi(b)H^{2}(\mathbb{T})$
has finite dimension. In fact, if $b$ is a finite Blaschke product, then
$\mathcal{D}$ has dimension equal to the number of zeros of $b$ and its
elements
are rational functions with poles located in a finite set outside the closed
unit disc. This may be seen by writing
\begin{equation}
b(z)=\prod_{j=1}^{N}b_{\alpha_{j}},\label{eq:Blaschke Product}\end{equation}
 where the $\alpha_{i}$ are the not-necessarily-distinct zeros of
$b$.  One can check that the functions $\{w_{i}\}_{i=1}^{N}$ constructed from
partial products of $b$ by way of
$$
w_{j}(z)=\frac{(1-\vert\alpha_{j}\vert^{2})^{1/2}}{1-\overline{\alpha_{j}}z}\prod_{k=1}^{j-1}b_{\alpha_{k}}, \qquad 1 \leq j \leq N,
$$
(the product $\prod_{k=1}^{j-1}b_{\alpha_{k}}$ is interpreted as $1$ when
$j=1$), form an orthonormal basis for $\mathcal{D}$ (see \cite[p. 305]{jW56}). 
We call this the \emph{canonical} orthonormal basis for $\mathcal{D}$.  Note
that the elements of the canonical basis are nonzero on $\T$ and hence
invertible elements of $C(\T)$.  The analysis in \cite{jW56} shows that if $b$
is not a finite Blaschke product, then $\mathcal{D}$ is infinite dimensional.
Alternatively, one may use the simple corollary of Beurling's theorem that
asserts that $\pi(\theta_1)H^2(\T) \subseteq \pi(\theta_2)H^2(\T)$ if and only
if the quotient $\theta_1/\theta_2$ is an inner function. (See \cite[page
11 ff.]{hH64}.) The point from this perspective is: if $b$ is not a finite
Blaschke product, then $b$ has infinitely many inner factors, say
$b=\Pi_{n=1}^{\infty} b_n$, and from these, one can construct an infinite
increasing sequence of closed subspaces of $\mathcal{D}$.
\end{remark}

To identify all the solutions to equation~\eqref{eq:cov1} in Problem
\ref{Problem: Central problem}, we need to restrict attention to
finite Blaschke products. For this reason and to get a clearer picture
of the Cuntz isometries implementing $\alpha$ and $\alpha_{+}$ we
emphasize: 
\begin{quotation}
\emph{From now on, $b$ will denote a} \textbf{finite} \emph{Blaschke
product.} 
\end{quotation}

In \cite[Theorem 1]{jR66}, Ryff shows that if $\varphi$ is analytic on the disc $\mathbb{D}$ and maps $\mathbb{D}$
into $\mathbb{D}$, then composition with $\varphi$ induces a bounded
operator on all the $H^{p}$ spaces. The the principal ingredient
in his proof is Littlewood's subordination theorem. In \cite[Theorem 3]{jR66}, Ryff shows further that composition with $\varphi$ is an isometry on $H^{p}$ if and only if $\varphi$ is an inner function
that vanishes at the origin. The following consequence of Theorem \ref{Thm: Main1} is a variation on this theme with a very
elementary proof.

\begin{corollary}\label{cor: boundedness gamma-b}Let $b$ be a finite
Blaschke product and define $\Gamma_{b}$ on trigonometric polynomials
$p$ by $\Gamma_{b}(p):=p\circ b$. Then $\Gamma_{b}$ extends in a unique way to
a bounded operator on $L^{2}(\mathbb{T})$ that leaves $H^{2}(\mathbb{T})$
invariant.  

Moreover, letting $\Gamma_{b}$ now denote the extension, the following are equivalent:
\begin{enumerate}
\item\label{gammaiso} $\Gamma_b$ is an isometry.
\item\label{bzero} $b(0)=0$.
\item\label{gammareduce} $\Gamma_b$ is reduced by $H^2(\T)$.
\end{enumerate}
\end{corollary}

\begin{proof} Fix an element $w$ of the canonical basis for $\mathcal{D} = H^2(\T) \ominus \pi(b) H^2(\T)$.  By Theorem~\ref{Thm: Main1} there is a unique isometry $S$ on $L^2(\T)$ satisfying
\begin{equation}\label{cuntzprop}
S(e_n) = w b^n, \qquad n \in \Z.
\end{equation}
As observed in Remark~\ref{Canonical Basis}, $w$ is an invertible element of $C(\T)$, so the operator $\pi(w)$ is invertible.  The relation \eqref{cuntzprop} then implies that for any trigonometric polynomial $p$ we have
$$
\pi(w^{-1}) S (p) = \Gamma_b (p),
$$
so the bounded operator $\pi(w^{-1}) S$ is an extension of $\Gamma_b$ to all of $L^2(\T)$.  Uniqueness of the extension follows from the density of the trigonometric polynomials in $L^2(\T)$.  The fact that $\Gamma_b(e_n) = b^n$ is in $H^2(\T)$ for every $n \geq 0$ implies that this extension leaves $H^2(\T)$ invariant.  

If $b(0)=0$, then $b(z)=zb_{1}(z)$, where $b_{1}$ is in $H^2(\T)$.  It follows that for any $n > m$ we have $(b^{n},b^{m})=(z^{n-m}b_{1}^{n-m},1)=0$, so that the family $\{b^n\}_{n \in \Z}$ is orthonormal.  Since $\Gamma_b(e_{n})=b^n$ for all $n \in \Z$, we conclude that $\Gamma_{b}$ is an isometry.   Conversely, if $\Gamma_{b}$ is an isometry,
$$
b(0)=(b,e_{0})=(\Gamma_{b}(e_{1}),\Gamma_{b}(e_{0}))=(e_{1},e_{0})=0.
$$
This establishes the equivalence of \eqref{gammaiso} and \eqref{bzero}.  

It will be useful later to deduce the equivalence of \eqref{bzero} and \eqref{gammareduce} from the assertion that if a vector $\xi \in L^2(\T)$ has the property that the pointwise product $\xi b$ is in $H^2(\T)$, then $\Gamma_b^* \xi$ is in $H^2(\T)$ if and only if $(\xi b)(0) = 0$.  To prove this assertion, note that $\Gamma_b^* \xi$ is in $H^2(\T)$ if and only if $(\Gamma_b^* \xi, z^{-n}) = 0$ for all $n > 0$, and this is equivalent to 
$$
0 = (\xi, \Gamma_b(z^{-n})) = (\xi, b^{-n}) = (\xi b^n, 1) = ((\xi b) b^{n-1}, 1) = (\xi b)(0) b^{n-1}(0), \qquad n > 0,
$$
which is equivalent to $(\xi b)(0) = 0$.  It follows from this assertion that $\Gamma_b^* \xi \in H^2(\T)$ for all $\xi \in H^2(\T)$ if and only if $b(0) = 0$.
\end{proof}

All of our proofs to this point have used only elementary operator theory.  To go further, we require more detailed
information about finite Blaschke products.

\section{The Master Isometry}

The zeros of our finite Blaschke product $b$ will be written $\alpha_{1},\alpha_{2},\ldots,\alpha_{N}$,
and we abbreviate $b_{\alpha_{j}}$ by $b_{j}$.  As it was in Corollary~\ref{cor: boundedness gamma-b}, the bounded operator of composition by $b$ on $L^2(\T)$ is denoted $\Gamma_b$.

Although all Cuntz families $\{S_{i}\}_{i=1}^{N}$ that we constructed
in Theorem \ref{Thm: Main1} are closely linked to $\Gamma_{b}$,
$\Gamma_{b}$ is not quite the operator we want to work with. It turns
out that there is a single \emph{isometry} $C_{b}$, 
built canonically from $\Gamma_{b}$, that has the property that
\emph{every} Cuntz family $S=\{S_{i}\}_{i=1}^{N}$ satisfying equation~
\eqref{eq:Cuntz1a} can be expressed in terms of $C_{b}$.
More remarkably, $C_{b}$ is \emph{reduced} by $H^{2}(\T)$. We call
this isometry the \emph{master isometry determined by} $b$ (or by
the endomorphism $\beta$ induced by $b$.)

Much of the material below is contained in results
already in the literature (see in particular \cite{jMcD03} and \cite{HW08}).
But many calculations are done under the additional hypothesis
that $b(0)=0$, which we want specifically to avoid. In the interest
of keeping our treatment self-contained, we present all of the details.

\begin{lemma}\label{Lem: Postive} Define \begin{equation}
J_{0}(z):=\frac{b^{\prime}(z)z}{Nb(z)}.\label{eq:J-zero}\end{equation}
 Then the restriction of $J_{0}$ to $\mathbb{T}$ is a positive continuous
function and, in particular, is bounded away from zero. \end{lemma}

\begin{proof} Of course, $J_{0}$ is a rational function. What needs
proof is that on $\T$, $J_{0}$ is positive, non-vanishing and has
no poles. If $\alpha_{j}\neq0$, then \[
\frac{b_{j}'(z)}{b_{j}(z)}=\frac{1}{z}\frac{1-|\alpha_{j}|^{2}}{|\alpha_{j}-z|^{2}},\]
 while if $\alpha_{j}=0$, $\frac{b_{j}'}{b_{j}}(z)=\frac{1}{z}$.
In either case, $\frac{zb_{j}'(z)}{b_{j}(z)}$ is strictly positive
on $\T$. A short calculation shows that $J_{0}(z)=\frac{1}{N}\sum_{j=1}^{N}\frac{zb_{j}^{'}(z)}{b_{j}(z)}$, so the result follows.\end{proof}

The next lemma follows \cite[Lemma 1]{jMcD03} closely.

\begin{lemma} \label{Lem: Local homeo}There is an increasing homeomorphism
$\theta:[0,2\pi]\to[\theta(0),\theta(0)+N\cdot2\pi]$, where $e^{i\theta(0)}=b(1)$,
such that 
\begin{enumerate}
\item $b(e^{it})=e^{i\theta(t)}$. 
\item The derivative of $\theta$ on $(0,2\pi)$ is $\frac{b'(e^{it})}{b(e^{it})}e^{it}\gneq0$. 
\item If $(t_{j-1},t_{j})=\theta^{-1}(\theta(0)+(j-1)\cdot2\pi,\theta(0)+j\cdot2\pi)$,
$j=1,2,\ldots,N$, and if $A_{j}:=\{e^{it}\mid t_{j-1}<t<t_{j}\}$,
then $\cup_{j=1}^{N}A_{j}=\T$, except for a finite set of points,
and $b$ maps each $A_{j}$ diffeomorphically onto $\mathbb{T}\backslash\{b(1)\}$. 
\item If $\sigma_{j}:\T\backslash\{b(1)\}\to A_{j}$ denotes the inverse
of the restriction of $b$ to $A_{j}$, then as $s$ ranges over $(\theta(0)+2\pi(j-1),\theta(0)+2\pi j)$,
$e^{is}$ ranges over $\T\backslash\{b(1)\}$ and \[
\sigma_{j}(e^{is})=e^{i\theta^{-1}(s)}.\]

\end{enumerate}
\end{lemma}

\begin{proof} Each $b_{j}$ is analytic in a neighborhood of the
closed unit disc and maps $\mathbb{T}$ homeomorphically onto $\mathbb{T}$
in an orientation preserving fashion. If the plane is slit along the
ray through the origin and $b_{j}(1)$, then one can define an analytic
branch of $\log z$ in the resulting region. On $\mathbb{T}\backslash\{b_{j}(1)\}$,
$\log b_{j}(e^{it})=i\theta_{j}(t)$ for a smooth function $\theta_{j}(t)$
defined initially on $(0,2\pi)$, and mapping to $(\theta_{j}(0),\theta_{j}(0)+2\pi)$.
Further, if one differentiates the defining equation for $\theta_{j}$,
one finds that $i\theta_{j}'(t)=\frac{b_{j}'(e^{it})}{b_{j}(e^{it})}e^{it}i$,
so $\theta_{j}^{'}$ is strictly positive, as was shown in the preceding
lemma. Hence $\theta_{j}$ is strictly increasing. Since $b_{j}(e^{i0})=b_{j}(1)=b_{j}(e^{i2\pi})$,
$\theta_{j}$ extends to a homeomorphism from $[0,2\pi]$ \emph{onto}
$[\theta_{j}(0),\theta_{j}(0)+2\pi]$. If $\theta$ is defined on
$[0,2\pi]$ by the formula $\theta(t):=\sum_{j=1}^{N}\theta_{j}(t)$,
then $\theta$ is a strictly increasing homeomorphism from $[0,2\pi]$
onto $[\theta(0),\theta(0)+N\cdot2\pi]$ such that $b(e^{it})=e^{i\theta(t)}$.
The remaining assertions are now clear. \end{proof}

\begin{definition}\label{definition: canonical transfer operator}
The \emph{(canonical) transfer operator} determined by the Blaschke
product $b$ is defined on measurable functions $\xi$ by the formula

\[
\Lb(\xi)(z):=\frac{1}{N}\sum_{b(w)=z}\xi(w).\]
 \end{definition}

Of course, an alternate formula for $\Lb$ is $\Lb(\xi)(z)=\frac{1}{N}\sum_{j=1}^{N}\xi(\sigma_{j}(z))$,
when $z\in\T\backslash\{b(1)\}$. It is clear that $\Lb$ carries
measurable functions to measurable functions, preserves order, and
is unital. Because $b$ is a local homeomorphism, $\Lb$ carries $C(\T)$
into itself. It is not difficult to see that $\Lb$ is a bounded linear
operator on $L^{2}(\T)$. However, we present a proof of this that connects
$\Lb$ with the adjoint of $\Gamma_{b}$. For this purpose, note that
by Lemma \ref{Lem: Postive}, $\pi(J_{0})$ is a bounded, positive,
invertible operator on $L^{2}(\T)$ with inverse $\pi(J_{0}^{-1})$.

\begin{theorem}\label{Theorem: transfer}\begin{equation}
\Lb\pi(J_{0})^{-1}=\Gamma_{b}^{*}\label{eq:Transfer}\end{equation}
 \end{theorem}

\begin{proof}For $\xi$ and $\eta$ in $L^{2}(\T)$, \[
(\Gamma_{b}^{*}\xi,\eta)=(\xi,\Gamma_{b}\eta)=\int_{0}^{2\pi}\xi(z)\overline{\eta(b(z))}\, dm(z)=\sum_{j=1}^{N}\int_{A_{j}}\xi(z)\overline{\eta(b(z))}\, dm(z).\]
 From the first and third assertions of Lemma \ref{Lem: Local homeo},
\[
\int_{A_{j}}\xi(z)\overline{\eta(b(z))}\, dm(z)=\int_{t_{j-1}}^{t_{j}}\xi(e^{it})\overline{\eta(e^{i\theta(t)})}\, dt.\]
Changes the variable to $s=\theta(t)$, the third and fourth assertions of Lemma \ref{Lem: Local homeo} imply
\[
\int_{t_{j-1}}^{t_{j}}\xi(e^{it})\eta(e^{i\theta(t)})\, dt=\int_{\theta(t_{j-1})}^{\theta(t_j)}\xi(e^{i\theta^{-1}(s)})\eta(e^{is})(\theta^{-1})'(s)\, ds.\]
Calculating $(\theta^{-1})'(s)=(\theta'(t))^{-1}=(\theta'(\theta^{-1}(s)))^{-1}$ and using the second assertion of Lemma \ref{Lem: Local homeo}, we deduce 
\[
\int_{\theta(t_{j-1})}^{\theta(t_j)}\xi(e^{i\theta^{-1}(s)})\eta(e^{is})(\theta^{-1})'(s)\, ds=\int_{\theta(t_j)}^{\theta(t_j)}\xi(e^{i\theta^{-1}(s)})\eta(e^{is})\frac{b(e^{i\theta^{-1}(s)})}{b'(e^{i\theta^{-1}(s)})e^{i\theta^{-1}(s)}}\, ds.\]
 But by the fourth statement of Lemma \ref{Lem: Local homeo} $e^{i\theta^{-1}(s)}=\sigma_{j}(e^{is})$,
when $s\in(\theta(0)+2\pi(j-1),\theta(0)+2\pi j)$. So the last integral is 
\[
\int_{\theta(t_{j-1})}^{\theta(t_j)}\xi(\sigma_{j}(e^{is}))\eta(e^{is})\frac{b(\sigma_{j}(e^{is}))}{b'(\sigma_{j}(e^{is}))\sigma_{j}(e^{is})}\, ds.\]
As $e^{is}$ sweeps out $\mathbb{T}\backslash\{b(1)\}$
as $s$ ranges over each interval $(\theta(t_{j-1}), \theta(t_j)) = (\theta(0)+2\pi(j-1),\theta(0)+2\pi j)$,
we conclude that \begin{eqnarray*}
(\Gamma_{b}^{*}\xi,\eta) & = & \sum_{j=1}^{N}\int_{\theta(0)+2\pi(j-1)}^{\theta(0)+2\pi j}\xi(\sigma_{j}(e^{is}))\eta(e^{is})\frac{b(\sigma_{j}(e^{is}))}{b'(\sigma_{j}(e^{is}))\sigma_{j}(e^{is})}\, ds\\
 & = & \frac{1}{N}\sum_{j=1}^{N}\int_{\T}\xi(\sigma_{j}(z))\overline{\eta(z)}\frac{Nb(\sigma_{j}(z))}{b'(\sigma_{j}(z))\sigma_{j}(z)}\, dm(z)\\
 & = & \frac{1}{N}\sum_{j=1}^{N}\int_{\T}\xi(\sigma_{j}(z))(J_{0}(\sigma_{j}(z)))^{-1}\overline{\eta(z)}\, dm(z)\\
 & = & (\Lb(\pi(J_{0})^{-1}\xi),\eta),\end{eqnarray*}
showing that $\Gamma_{b}^{*}=\Lb\pi(J_{0})^{-1}$. \end{proof}

\begin{notation}\label{notation: J}\[
J(z):=\exp\left[\frac{1}{2\pi}\int_{-\pi}^{\pi}\frac{e^{it}+z}{e^{it}-z}\ln(J_{0}(e^{it}))dt\right]\]
 \end{notation} Of course $J$ is the unique outer function that
is positive at $0$ and satisfies the equation $\vert J(z)\vert=J_{0}(z)$
for all $z\in\T$. (See \cite[Theorem 5]{hH64} and the surrounding
discussion.)  Significantly, $J$ does not vanish on $\D$ and $J$ is in $H^{\infty}(\T)$; note that $J_0$ is not even in $H^2(\T)$ except in trivial cases.  We will work primarily with $J^{\frac{1}{2}}$, which
is $\exp\left[\frac{1}{4\pi}\int_{-\pi}^{\pi}\frac{e^{it}+z}{e^{it}-z}\ln(J_{0}(e^{it}))dt\right]$.  Note that $J^{1/2}$ and $J^{-1/2}$ are both in $H^{\infty}(\T)$.

\begin{lemma}\label{lemma:Proto-Conjugate}For all $\varphi\in L^{\infty}(\T)$,\[
\Lb\pi(\varphi)\Gamma_{b}=\pi(\Lb(\varphi)).\]
 In particular, $\Lb$ is a left inverse for $\Gamma_{b}$. \end{lemma}

\begin{proof}Take $\xi\in L^{2}(\T)$ and $\varphi\in L^{\infty}(\T)$
and calculate:
\begin{align*}
\Lb(\pi(\varphi)\Gamma_{b}(\xi))(z)& =\frac{1}{N}\sum_{b(w)=z}(\pi(\varphi)\Gamma_{b}(\xi))(w)=\frac{1}{N}\sum_{b(w)=z}\varphi(\omega)\xi(b(w))\\
& =\frac{1}{N}\sum_{b(w)=z}\varphi(\omega)\xi(z)=\Lb(\varphi)(z)\xi(z)\\
& =(\pi(\Lb(\varphi))\xi)(z).
\end{align*}
 \end{proof}

\begin{lemma}\label{lemma:Cbiso}Set \[
C_{b}:=\pi(J^{\frac{1}{2}})\Gamma_{b}.\]
 Then $C_{b}$ is an isometry on $L^{2}(\T)$ and \[
C_{b}^{*}=\Lb\pi(J^{-\frac{1}{2}}).\]
 Further, if $\{S_{i}\}_{i=1}^{N}$ is the Cuntz family constructed
in Theorem \ref{Thm: Main1} using an orthonormal basis $\{v_{i}\}_{i=1}^{N}$
for $H^{2}(\T)\ominus\pi(b)H^{2}(\T)$, then $S_{i}=\pi(v_{i}J^{-\frac{1}{2}})C_{b}$ for all $1 \leq i \leq N$.\end{lemma}

\begin{proof} The key is the relation $\Lb \pi(J_0^{-1}) = \Gamma_b^*$ from Theorem~\ref{Theorem: transfer}.  We just compute:
\begin{equation*}
C_{b}^{*}C_{b} = \Gamma_{b}^{*}\pi(\overline{J^{\frac{1}{2}})}\pi(J^{\frac{1}{2}})\Gamma_{b} = \Gamma_{b}^{*}\pi(\vert J\vert)\Gamma_{b} =  \Gamma_{b}^{*}\pi(J_{0})\Gamma_{b} = \Lb\Gamma_{b}=I.
\end{equation*}
and
\begin{equation*}
C_{b}^{*} = \Gamma_{b}^{*}\pi(\overline{J^{\frac{1}{2}}}) 
= \Lb\pi(J_{0})^{-1}\pi(\overline{J^{\frac{1}{2}}}) 
=  \Lb\pi(J^{-\frac{1}{2}}).
\end{equation*}
For the final assertion, simply observe that the definition of $S_{i}$
(using $\{v_{i}\}_{i=1}^{N})$ shows that $S_{i}=\pi(v_{i})\Gamma_{b}$.  As $C_{b}=\pi(J^{\frac{1}{2}})\Gamma_{b}$,
we conclude 
\begin{eqnarray*}
S_{i} & = & \pi(v_{i})\pi(J^{-\frac{1}{2}})\pi(J^{\frac{1}{2}})\Gamma_{b}=\pi(v_{i}J^{-\frac{1}{2}})C_{b}.\\
\\\end{eqnarray*}
 \end{proof}

\begin{proposition}\label{Reducing and implementing L}$H^{2}(\T)$
reduces $C_{b}$ and $C_{b}$ implements $\Lb$ in the sense that
\begin{equation}
C_{b}^{*}\pi(\varphi)C_{b}=\pi(\Lb(\varphi)),\label{eq: Implement L}\end{equation}
 for all $\varphi\in L^{\infty}(\T)$. \end{proposition}

\begin{proof}Since $\Gamma_{b}$ and $\pi(J^{\frac{1}{2}})$ leave
$H^{2}(\T)$ invariant, so does $C_{b}=\pi(J^{\frac{1}{2}})\Gamma_{b}$.
On the other hand, $C_{b}^{*}=\Lb\pi(J^{-\frac{1}{2}})$ by Lemma
\ref{lemma:Cbiso}, so one way to show that $H^2(\T)$ reduces $C_b$ is to show that $\Lb$ leaves
$H^{2}(\T)$ invariant. McDonald did this in \cite[Lemma 2]{jMcD03}.  

We can also prove this directly: fixing $\eta \in H^2(\T)$, we must show that $\Lb \eta \in H^2(\T)$.  By Theorem~\ref{Theorem: transfer} we have that $\Lb = \Gamma_b^* \pi(J_0)$, so it suffices to show that the vector $\xi = \pi(J_0) \eta \in L^2(\T)$ is mapped into $H^2(\T)$ by $\Gamma_b^*$.  By the definition \eqref{eq:J-zero} of $J_0$ we have
$$
b(z) \xi(z) = b(z) J_0(z) \eta(z) = z b'(z) \eta(z),
$$
showing that $b \xi$ is in $H^2(\T)$ and that $(b \xi)(0) = 0$.  Thus $\Gamma_b^* \xi \in H^2(\T)$ by the argument given in the proof of Corollary~\ref{cor: boundedness gamma-b}.

Equation~\eqref{eq: Implement L} follows from Lemmas \ref{lemma:Cbiso}
and \ref{lemma:Proto-Conjugate} because \[
C_{b}^{*}\pi(\varphi)C_{b}=\Lb\pi(J^{-\frac{1}{2}})\pi(\varphi)\pi(J^{\frac{1}{2}})\Gamma_{b}=\Lb\pi(\varphi)\Gamma_{b}=\pi(\Lb(\varphi)).\]
 \end{proof}

We shall denote the restriction of $C_{b}$ to $H^{2}(\T)$ by $C_{b+}$.

\begin{theorem}\label{Thm: Intertwine}
\begin{enumerate}
\item \label{first} If $T$ is a bounded operator on $L^{2}(\T)$, then $T$ satisfies 
\begin{equation}
T\pi(\varphi)=\pi(\beta(\varphi))T\label{eq:Intertwine V}, \qquad \varphi \in L^{\infty}(\T),
\end{equation}
if and only if $T=\pi(m)C_{b}$ for some function $m\in L^{\infty}(\T)$.
\item \label{second} If $T$ is a bounded operator on $H^{2}(\T)$, then \begin{equation}
T\tau(\varphi)=\tau(\beta(\varphi))T\label{eq: Intertwine+}, \qquad \varphi \in H^{\infty}(\T), \end{equation}
if and only if $T=\tau(m)C_{b+}$ for some function $m\in H^{\infty}(\T)$.
\end{enumerate}
Further, if $T=\pi(m)C_{b}$, $m\in L^{\infty}(\T)$, (resp. if $T=\tau(m)C_{b+}$,
$m\in H^{\infty}(\T)$) then $\Vert T\Vert=\left(\Vert\mathcal{L}(|m|^{2})\Vert_{\infty}\right)^{\frac{1}{2}}$,
and $T$ is an isometry if and only if $\mathcal{L}(|m|^{2})=1$ a.e.
\end{theorem}

\begin{proof} We begin by proving assertion \eqref{first}.  If $T=\pi(m)C_{b}$ for some $m\in L^{\infty}(\T)$,
a short calculation shows that $T$ satisfies \eqref{eq:Intertwine V}.  The formula~\eqref{eq: Implement L} then implies\[
T^{*}T=C_{b}^{*}\pi(\overline{m})\pi(m)C_{b}=\pi(\Lb(|m|^{2})),\]
and $\Vert T\Vert=\left(\Vert\mathcal{L}(|m|^{2})\Vert_{\infty}\right)^{\frac{1}{2}}$ follows as $\pi$ is faithful.  The fact that $T$ is isometric if and only if $\mathcal{L}(|m|^{2})=1$ a.e. is immediate.

Suppose conversely that $T$ is an operator on $L^2(\T)$ satisfying \eqref{eq:Intertwine V}. Define $m:=\pi(J^{-1/2})T(\mathbf{1})$,
where $\mathbf{1}$ is the constant function that is identically equal to $1$.  Note that a priori we have that $m \in L^2(\T)$, but not that $m \in L^{\infty}(\T)$.

The hypothesis \eqref{eq:Intertwine V} and the definition of $m$ imply that
\begin{equation}
T\varphi=T\pi(\varphi)\mathbf{1}=\pi(\beta(\varphi))\pi(J^{1/2})m=mC_{b}(\varphi),\label{eq:InitPropsofm} \qquad \varphi \in L^{\infty}(\T).
\end{equation}
If more generally $\varphi \in L^2(\T)$, there is a sequence $\varphi_n$ in $L^{\infty}(\T)$ such that $\varphi_n \to \varphi$ in $L^2(\T)$.  Boundedness of $C_b$ implies that the sequence of vectors $C_b \varphi_n$ is convergent in $L^2(\T)$ with limit $C_b \varphi$, and boundedness of $T$ together with \eqref{eq:InitPropsofm} implies that the sequence $m C_b \varphi_n$ is convergent in $L^2(\T)$ with limit $T \varphi$.  By passing to a subsequence as necessary we may assume that $C_b \varphi_n \to C_b \varphi$ pointwise a.e., and $m C_b \varphi_n \to T \varphi$ pointwise a.e, and deduce $T \varphi = m C_b \varphi$.  We conclude that 
\begin{equation}
T\varphi=mC_{b}(\varphi),\label{eq:l2holds} \qquad \varphi \in L^2(\T).
\end{equation}

Fix an orthonormal basis $\{v_{i}\}_{i=1}^{N}$ for $H^{2}(\T)\ominus\pi(b)H^{2}(\T)$ and set $S_{i}=\pi(v_{i}J^{-\frac{1}{2}})C_{b}$.  By Lemma \ref{lemma:Cbiso} and Theorem \ref{Thm: Main1}, $\{S_{i}\}_{i=1}^{N}$
is a Cuntz family, so for any $\xi \in L^2(\T)$ we have 
\begin{align*}
m\xi =  m\sum_{j=1}^{N}S_{j}S_{j}^{*}\xi & =  m\sum_{j=1}^{N}\pi(v_{j}J^{-1/2})C_{b} S_j^{*}\xi \\
 & = \sum_{j=1}^{N} v_{j}J^{-1/2} m C_{b} S_j^{*}\xi \\
 & = \sum_{j=1}^N v_j J^{-1/2} T S_j^* \xi & \text{by \eqref{eq:l2holds}.}
\end{align*}
Thus multiplication by $m$ is the operator $\sum_{j=1}^N \pi(v_j J^{-1/2}) T S_j^*$ on $L^2(\T)$.  As this operator is bounded we deduce that $m \in L^{\infty}(\T)$ as desired.

The proof of assertion \eqref{second} is similar, but it is important to keep track of the differences. If $T=\tau(m)C_{b+}$ for some $m\in H^{\infty}(\T)$, it is easily seen that \eqref{eq: Intertwine+} is satisfied, since $\tau(m)$ and $\tau(\varphi)$
commute when $m$ and $\varphi$ are in $H^{\infty}(\T)$, and since
$C_{b+}$ is the restriction of $C_{b}$ to a reducing subspace. Furthermore,
\[
T^{*}T=C_{b+}^{*}\tau(m)^{*}\tau(m)C_{b+}=PC_{b}^{*}P\pi(\overline{m})P\pi(m)PC_{b}P\vert_{H^{2}(\T)}.\]
 Since $H^{2}(\T)$ is invariant under $\pi(m)$ and reduces $C_{b}$, we deduce
\[
T^* T = PC_{b}^{*}\pi(\overline{m})\pi(m)C_{b}P\vert_{H^{2}(\T)}=P\pi(\Lb(\vert m\vert^{2}))P\vert_{H^{2}(\T)}=\tau(\Lb(\vert m\vert^{2})).\]
Thus $\Vert T^{*}T\Vert=\Vert\tau(\Lb(\vert m\vert^{2}))\Vert=\Vert\Lb(\vert m\vert^{2})\Vert_{\infty}$,
which proves the formula for the norm of $T$. Also, it shows that
$T$ is an isometry if and only if $\mathcal{L}(|m|^{2})=1$ a.e.

Suppose conversely that $T$ on $H^{2}(\T)$ satisfies equation~\eqref{eq: Intertwine+} and set $m:=\tau(J^{-1/2})T(\mathbf{1})$; we know $m \in H^2(\T)$ and wish to deduce that $m \in H^{\infty}(\T)$.  The fact that $J^{-\frac{1}{2}}\in H^{\infty}(\T)$ and the properties
of $C_{b+}$ show that \[
T\varphi=T\tau(\varphi)\mathbf{1}=\tau(\beta(\varphi))\tau(J^{1/2})m=mC_{b+}(\varphi)\]
 for all $\varphi\in H^{\infty}(\T)$ and hence all $\varphi \in H^2(\T)$.  With $S_{i}=\pi(v_{i}J^{-\frac{1}{2}})C_{b}$ as before, we note that $H^{2}(\T)$ reduces $S_{i}$, by Theorem \ref{Thm: Main1}, and we set $R_{i}:=S_{i}\vert_{H^{2}(\T)}$. Theorem \ref{Thm: Main1} asserts that $\{R_{i}\}_{i=1}^{N}$ is a Cuntz family
of isometries on $H^{2}(\T)$. Since $H^{2}(\T)$ reduces $C_{b}$,
we have for any $\xi \in H^2(\T)$
\begin{align*}
m\xi & = m\sum_{j=1}^{N}R_{j}R_{j}^{*}\xi = m\sum_{j=1}^{N} \pi(v_{j}J^{-1/2})C_{b}R_j^* \xi 
  = \sum_{j=1}^{N} v_{j}J^{-1/2} mC_{b+}R_{j}^{*}\xi \\
& = \sum_{j=1}^{N} \pi(v_{j}J^{-1/2}) T R_{j}^{*}\xi. 
\end{align*}
As $v_j J^{-1/2} \in H^{\infty}$ for each $j$, the conclusion is that multiplication by $m$ is the bounded operator $\sum_{j=1}^n \tau(v_j J^{-1/2}) T R_j^*$ on $H^2(\T)$.  Thus $m\in H^{\infty}(\T)$.
\end{proof}

We have called $C_{b}$ \emph{the} master isometry.  One explanation of our use of the definite article is that when one builds the Deaconu-Renault groupoid $G$ determined by $b$, viewed as a local homeomorphism of $\T$, then $C_{b}$ appears as the image of a special isometry $S$ in the groupoid $C^{*}$-algebra $C^{*}(G)$ under a representation that gives rise to the Cuntz familes we consider here.  We have not seen any compelling reason to bring this technology into this note - nevertheless, $C^{*}(G)$ and $S$ are lying in the background and may prove useful in the future.  For further information about the use of groupoids and $C^{*}$-algebras generated by local homeomorphisms, see \cite{IM08}.

One should not infer from our use of the definite article that $C_b$ is uniquely determined by the abstract properties that we have shown it has.  More precisely, suppose that $V$ is an isometry on $L^2(\T)$ that implements $\Lb$ in the sense of \eqref{eq: Implement L}, interacts with $\pi$ as in \eqref{eq:Intertwine V}, and is reduced by $H^2(\T)$.  By Theorem~\ref{Thm: Intertwine}, $V$ must be of the form $V=\pi(m)C_{b}$ for some $m\in L^{\infty}(\T)$ satisfying $\Lb(|m|^{2})=1$. The assumption that $V$ is reduced by $H^{2}(\T)$, together with the assumption that $V$ implements $\Lb$ imply that $|m|=1$ a.e.; it may further be shown that $m$ is an inner function with the property that $\Lb(\overline{m})$ is
constant. But in general we have been unable to deduce more about $m$ than this.  (We note that $m$ need not be constant.  If $b(z)=z^{2}$, then $V=\pi(m)C_{b}$ will have all of the indicated properties if $m(z) = z^k$ for any odd positive integer $k$.  What happens when $b$ is a more general Blaschke product remains to be investigated.)

\section{Hilbert modules and orthonormal bases}

The endomorphism $\beta$ of $L^{\infty}(\T)$ and the transfer operator
$\Lb$ may be used to endow $L^{\infty}(\T)$ with the structure of
a Hilbert $C^{*}$-module over $L^{\infty}(\T)$. We will exploit
this structure in order to solve Problem \ref{Problem: Central problem}.
We do not need much of the general theory about these modules. Rather,
we only need to expose enough so that the formulas we use make good
sense. Excellent references for the basics of the theory are \cite{cL95,MT05}.

Suppose $A$ is a $C^{*}$-algebra and that $E$ is a right $A$-module.
Then $E$ is called a \emph{Hilbert $C^{*}$-module} over $A$ in
case $E$ is endowed with an $A$-valued sesquilinear form $\<\cdot,\cdot\>:E\times E\to A$
that is subject to the following conditions. 
\begin{enumerate}
\item $\<\cdot,\cdot\>$ is conjugate linear in the first variable,
so $\<\xi\cdot a,\eta\cdot b\>=a^{*}\<\xi,\eta\> b$. 
\item For all $\xi\in E$, $\<\xi,\xi\>$ is a positive element
in $A$ that is $0$ if and only if $\xi=0$. 
\item $E$ is complete in the norm defined by the formula $\Vert\xi\Vert_{E}:=\Vert\<\xi,\xi\>\Vert_{A}^{\frac{1}{2}}$. 
\end{enumerate}
Of course, it takes a little argument to prove that $\Vert\cdot\Vert_{E}$
is a norm on $E$.

In the application of Hilbert modules that we have in mind, our $C^{*}$-algebra
$A$ will be unital, and we will denote the unit by $\mathbf{1}$.
A vector $v\in E$ is called a \emph{unit vector} if $\< v,v\>=\mathbf{1}$.
Note that this says more than simply $\Vert v\Vert=1$. A family $\{v_{i}\}_{i\in I}$
of vectors in $E$ is called is called an \emph{orthonormal set} if
$\< v_{i},v_{j}\>=\delta_{ij}\mathbf{1}$. Further, if linear
combinations of vectors from $\{v_{i}\}_{i\in I}$ (where the coefficients
are from $A$) are dense in $E$ then we say that $\{v_{i}\}_{i\in I}$
is an \emph{orthonormal basis} for $E$. In this event, every vector
$\xi\in E$ has the representation \begin{equation}
\xi=\sum_{i\in I}v_{i}\cdot\< v_{i},\xi\>,\label{eq:FourierExpansion}\end{equation}
 where the sum converges in the norm of $E$. In general, a Hilbert
$C^{*}$-module need not have an orthonormal basis. Also, in general,
two orthonormal bases need not have the same cardinal number. Nevertheless,
two orthonormal bases $\{v_{i}\}_{i\in I}$ and $\{w_{j}\}_{j\in J}$
are linked by a unitary matrix over $A$ in the usual way:\[
w_{j}=\sum_{i\in I}v_{i}\cdot\< v_{i},w_{j}\>=\sum_{i\in I}v_{i}\cdot u_{ij}.\]
 So if the cardinality of $I$ is $n$ and the cardinality of $J$
is $m$, then $U=(u_{ij})$ is a unitary matrix in $M_{nm}(A)$, i.e.,
$UU^{*}=\mbox{\textbf{1}}_{n}$ in $M_{n}(A)$, while $U^{*}U=\mathbf{1}_{m}$
in $M_{m}(A)$. And conversely, any such matrix transforms the orthonormal
basis $\{v_{i}\}_{i\in I}$ for $E$ into an orthonormal basis $\{w_{j}\}_{j\in J}$
for $E$ via this formula. In our application of these notions, the
coefficient algebra $A$ will be commutative so, as is well known,
all unitary matrices are square and, therefore, any two orthonormal
bases have the same number of elements.

We shall view $L^{\infty}(\T)$ as right module over $L^{\infty}(\T)$
via the formula \begin{equation}
\xi\cdot a:=\xi\beta(a),\label{eq:module} \qquad  a, \xi \in L^{\infty}(\T),
\end{equation}
where the product on the right hand side is the usual pointwise product in $L^{\infty}(\T)$.
Also, we shall use $\Lb$ to endow $L^{\infty}(\T)$ with the $L^{\infty}(\T)$-valued
inner product defined by the formula \begin{equation}
\<\xi,\eta\>:=\Lb(\overline{\xi}\eta),\label{eq:innerproduct} \qquad \xi,\eta\in L^{\infty}(\T). \end{equation}
Using the fact that $\Lb\circ\beta=\id_{L^{\infty}(\T)}$
(Lemma \ref{lemma:Proto-Conjugate}), it is straightforward to see
that $L^{\infty}(\T)$ is a Hilbert $C^{*}$-module over $L^{\infty}(\T)$,
which we shall denote by $\LiL$. The only thing that may be seem
problematic is the fact that $\LiL$ is complete in the norm defined
by the inner product. However, a moment's reflection reveals that the norm is equivalent to the $L^{\infty}(\T)$-norm, which is complete.  We remark that \eqref{eq:module} and \eqref{eq:innerproduct} make sense
when the functions in $L^{\infty}(\T)$ are restricted to lie in $C(\T)$,
and $C(\T)$ also is a Hilbert module over $C(\T)$ in this structure, but we will focus on the $L^{\infty}(\T)$ case in what follows.

Vectors $\{m_{i}\}_{i=1}^{N}$ in $\LiL$ form an orthonormal basis
for $\LiL$ if and only if \[
\Lb(\overline{m_{i}}m_{j})=\< m_{i},m_{j}\>=\delta_{ij}\mathbf{1},\]
 where $\mathbf{1}$ is the constant function $1$, and 
$$
f=\sum_{i=1}^{N}m_{i}\cdot\< m_{i},f\>=\sum_{i=1}^{N}m_{i}\beta(\Lb(\overline{m_{i}}f)), \qquad f \in \LiL.
$$
We have intentionally used $N$, the order of the Blaschke product
$b$, as the upper limit in these sums because $\LiL$ has an orthonormal
basis with $N$ elements, viz. $\{\sqrt{N}1_{A_{i}}\}_{i=1}^{N}$,
where the $A_{i}$'s are the arcs in Lemma \ref{Lem: Local homeo},
and because any two orthonormal bases for $\LiL$ have the same number
of elements, as we noted above.

\begin{remark}\label{rem:ConditionalExpectations} As a map on $L^{\infty}(\T)$,
$\mathbb{E}:=\beta\circ\Lb$ is the conditional expectation onto the
range of $\beta$. Indeed, $\mathbb{E}$ is a weak-$*$ continuous,
positivity preserving, idempotent unital linear map on $L^{\infty}(\T)$.
So it is the restriction to $L^{\infty}(\T)$ of an idempotent and contractive
linear map on $L^{1}(\T)$ that preserves the constant functions.
Hence $\mathbb{E}$ is a conditional expectation by the corollary
to \cite[Theorem 1]{rD65}. Of course, the range of $\mathbb{E}$
consists of functions in the range of $\beta$ by definition. On the
other hand, if $f=\beta(g)$ for some function $g\in L^{\infty}(\T)$,
then $\mathbb{E}(f)=\beta\circ\Lb\circ\beta(g)=\beta(g)=f$, since
$\Lb$ is a left inverse for $\beta$. Thus, to say that $\{m_{i}\}_{i=1}^{N}$
is an orthonormal basis for $\LiL$ is to say that 
$$
f=\sum_{i=1}^{N}m_{i}\mathbb{E}(\overline{m}_{i}f), \qquad f \in L^{\infty}(\T).
$$
\end{remark}

In light of the discussion in Section~\ref{sec:preliminaries}, the following describes all solutions to \eqref{eq:cov1}.

\begin{theorem}\label{Thm:Solution 1} If a Cuntz family $S = \{S_i\}_{i=1}^N$ on $B(L^2(\T))$ gives rise to a covariant representation $(\pi,\alpha_{S})$ of $(L^{\infty}(\T),\beta)$, then there is an orthonormal
basis $\lbrace m_{i}\rbrace_{i=1}^{N}$ for $L^{\infty}(\T)_{\mathcal{L}}$
such that \begin{equation}
S_{i}=\pi(m_{i})C_{b},\label{eq:DefS_i} \qquad 1 \leq i \leq N.\end{equation}
Further, if $\{m_{i}\}_{i=1}^{N}$ is any family
of functions in $L^{\infty}(\T)$ such that the operators $S_i$ defined by \eqref{eq:DefS_i} form a Cuntz family $S$ such that $(\pi, \alpha_S)$ is a covariant representation of $(L^{\infty}(\T), \beta)$, then $\{m_{i}\}_{i=1}^{N}$ is an orthonormal basis for $\LiL$. Conversely, if $\{m_{i}\}_{i=1}^{N}$ is any orthonormal basis for $\LiL$ and $S_i$ is defined by \eqref{eq:DefS_i} for $1 \leq i \leq N$, then $S=\{S_{i}\}_{i=1}^{N}$ is a Cuntz family and $(\pi, \alpha_S)$ is a covariant representation of $(L^{\infty}(\T), \beta)$. 
\end{theorem}

\begin{proof} Suppose $\lbrace S_{i}\rbrace_{i=1}^{N}$ is a Cuntz family
on $L^{2}(\T)$ that satisfies equation~\eqref{eq:Cuntz1a}. If both
sides of this equation are multiplied on the right by $S_{j}$, then
one finds from equation~\eqref{eq:Cuntz 1} that $S_{j}\pi(\cdot)=\pi\circ\beta(\cdot)S_{j}$
for each $j$. By Theorem \ref{Thm: Intertwine}, for each $j$ there is $m_j \in L^{\infty}(\T)$ satisfying $S_{j}=\pi(m_{j})C_{b}$ and $\mathbf{1}=\Lb(|m_{j}|^{2})=\< m_{j},m_{j}\>$.
The fact that $S$ satisfies equation~\eqref{eq:Cuntz 1} then yields
$$
\delta_{i,j}I_{L^{2}(\T)}  =  S_{i}^{*}S_{j}
=  C_{b}^{*}\pi(\overline{m_{i}}m_{j})C_{b}
=  \pi(\mathcal{L}(\overline{m_{i}}m_{j}))
= \pi(\< m_{i},m_{j}\>).
$$
 Since $\pi$ is faithful, $\< m_{i},m_{j}\>=\delta_{i,j}\mathbf{1}$,
where $\mathbf{1}$ is the constant function $1$. Thus, $\lbrace m_{i}\rbrace_{i=1}^{N}$
is an orthonormal set in $L^{\infty}(\T)_{\mathcal{L}}$. We now show that the $\{m_i\}_{i=1}^N$ span $\LiL$.  If 
$f\in\LiL$ satisfies $\< f,m_{i}\>=0$
for all $i$, then we have 
\begin{align*}
(\pi(f)C_{b})^{*} = C_{b}^{*}\pi(\overline{f})\left(\sum_{i=1}^{N}S_{i}S_{i}^{*}\right) & = C_{b}^{*}\pi(\overline{f})\left(\sum_{i=1}^{N}\pi(m_{i})C_{b}S_{i}^{*}\right)\\
 & = \sum_{i=1}^{N}C_{b}^{*}\pi(\overline{f})\pi(m_{i})C_{b}S_{i}^{*}\\
 & = \sum_{i=1}^{N}\pi(\< f,m_{i}\>)S_{i}^{*} & \text{by \eqref{eq: Implement L}} \\
 & = 0,
\end{align*}
and thus $\pi(f)C_{b}=0$, which in turn implies $fJ^{\frac{1}{2}}=\pi(f)J^{\frac{1}{2}}=\pi(f)C_{b}\mathbf{1}=0$, and thus $f=0$.  This shows that $\{m_{i}\}_{i=1}^{N}$ is an orthonormal basis for $\LiL$.

For the converse assertion, suppose $\lbrace m_{i}\rbrace_{i=1}^{N}$
is any orthonormal basis for $\LiL$, and set $S_{i}:=\pi(m_{i})C_{b}$.
Then from \eqref{eq: Implement L} we deduce \begin{eqnarray*}
S_{i}^{*}S_{j} & = & C_{b}^{*}\pi(\overline{m_{i}}m_{j})C_{b}\\
 & = & \pi(\< m_{i},m_{j}\>)\\
 & = & \delta_{i,j}\pi(\mathbf{1})=\delta_{i,j}I_{L^{2}(\T)}.\end{eqnarray*}
 So the relations \eqref{eq:Cuntz 1} are satisfied.  To verify the Cuntz identity \eqref{eq:Cuntz 2}, note first that equation~\eqref{eq:Cuntz 1} shows that the sum $\sum_{i=1}^{N}S_{i}S_{i}^{*}$ is a projection.
To show that $\sum_{i=1}^{N}S_{i}S_{i}^{*}=I$, it suffices
to show that $\sum_{i=1}^{N}S_{i}S_{i}^{*}$ acts as the identity
operator on a dense subset of $L^{2}(\T)$.  So fix $f\in L^{\infty}(\T)$ and observe
that we may write \begin{equation}
\sum_{i=1}^{N}S_{i}S_{i}^{*}f=\sum_{i=1}^{N}S_{i}S_{i}^{*}\pi(f)1=\sum_{i=1}^{N}S_{i}S_{i}^{*}\pi(f)\Gamma_{b}1=\sum_{i=1}^{N}S_{i}S_{i}^{*}\pi(fJ^{-\frac{1}{2}})C_{b}1.\label{eq:ONB1}\end{equation}
 Since $S_{i}=\pi(m_{i})C_{b}$, the last sum in \eqref{eq:ONB1}
is \[
\sum_{i=1}^{N}\pi(m_{i})C_{b}C_{b}^{*}\pi(\overline{m_{i}})\pi(fJ^{-\frac{1}{2}})C_{b}1=\sum_{i=1}^{N}\pi(m_{i})C_{b}\pi(\Lb(\overline{m_{i}}fJ^{-\frac{1}{2}}))1,\]
 where we have used \eqref{eq: Implement L}. But by Theorem \ref{Thm: Intertwine} the right hand side of this equation is \[
\sum_{i=1}^{N}\pi(m_{i})\pi(\beta(\Lb(\overline{m_{i}}fJ^{-\frac{1}{2}})))C_{b}1=\pi\left(\sum_{i=1}^{N}m_{i}\< m_{i},fJ^{-\frac{1}{2}}\>\right)C_{b}1=\pi(fJ^{-\frac{1}{2}})C_{b}1,\]
 because $\{m_{i}\}_{i=1}^{N}$ is an orthonormal basis for $\LiL$,
by hypothesis. As $C_{b}1=\pi(J^{\frac{1}{2}})\Gamma_{b}1=\pi(J^{\frac{1}{2}})1$ 
it follows that $\pi(fJ^{-\frac{1}{2}})C_{b}1=f$, and thus $\sum_{i=1}^{N}S_{i}S_{i}^{*}f=f$.  We conclude that $S = \{S_i\}_{i=1}^N$ is a Cuntz family.  

To see that this family implements $\beta$, simply note
that\[
\pi(\beta(\varphi))=\pi(\beta(\varphi))\sum_{i=1}^{N}S_{i}S_{i}^{*}=\sum_{i=1}^{N}S_{i}\pi(\varphi)S_{i}^{*}\]
 since the $S_{i}$ satisfy equation~\eqref{eq:Intertwine V}. \end{proof}

\begin{corollary}\label{Cor: module basis} If $\{v_{i}\}_{i=1}^{N}$
is an orthonormal basis for the Hilbert \emph{space} $H^{2}(\T)\ominus\pi(b)H^{2}(\T)$,
then the functions $\{v_{i}J^{-\frac{1}{2}}\}_{i=1}^{N}$ form an
orthonormal basis for the Hilbert \emph{module} $\LiL$.\end{corollary}

\begin{proof}By Lemma~\ref{lemma:Cbiso}, the Cuntz isometries coming from $\{v_i\}_{i=1}^N$ via Theorem \ref{Thm: Main1} have the form $\pi(v_{i}J^{-\frac{1}{2}})C_{b}$.
Therefore by Theorem \ref{Thm:Solution 1} the functions $\{v_{i}J^{-\frac{1}{2}}\}_{i=1}^{N}$
form an orthonormal basis for $\LiL$. \end{proof}

\begin{corollary}\label{cor: Uniqueness S} If $S^{(1)}$ and $S^{(2)}$
are two Cuntz families in $B(L^{2}(\T))$ satisfying
$$
\alpha_{S^{(i)}}\circ\pi=\pi\circ\beta, \qquad i = 1, 2,
$$
then there is a unitary matrix $(u_{ij})$ in $M_{N}(L^{\infty}(\T))$
such that \begin{equation}
S_{j}^{(2)}=\sum_{i=1}^{N}S_{i}^{(1)}\pi(u_{ij}),\label{eq:Uniqueness S}\end{equation}
 $j=1,2,\cdots,N$. Conversely, if $S^{(1)}$ and $S^{(2)}$ are Cuntz
families on $L^{2}(\T)$ that are linked by equation~\eqref{eq:Uniqueness S},
then $\alpha_{S^{(1)}}$ implements $\beta$ if and only if $\alpha_{S^{(2)}}$
implements $\beta$. Further, $\alpha_{S^{(1)}}=\alpha_{S^{(2)}}$
on $B(L^{2}(\T))$ if and only if $(u_{ij})$ is a unitary matrix
of constant functions. \end{corollary}

\begin{proof}By Theorem \ref{Thm:Solution 1}, we may suppose there
are orthonormal bases $\{m_{i}^{(1)}\}_{i=1}^{N}$ and $\{m_{i}^{(2)}\}_{i=1}^{N}$
for $\LiL$ that define $S^{(1)}$ and $S^{(2)}$. In this event,
there is a unitary matrix $(u_{ij})$ in $M_{N}(L^{\infty}(\T))$
so that \[
m_{j}^{(2)}=\sum_{i=1}^{N}m_{i}^{(1)}\cdot u_{ij}.\]
 But then we may use \eqref{eq:DefS_i} to derive \eqref{eq:Uniqueness S} as follows:
\begin{align*}
S_{j}^{(2)} & =  \pi(m_{j}^{(2)})C_{b}  = \sum_{i=1}^{N}\pi(m_{i}^{(1)})\pi(\beta(u_{ij}))C_{b}
=  \sum_{i=1}^{N}\pi(m_{i}^{(1)})C_{b}\pi(u_{ij})\\
& = \sum_{i=1}^{N}S_{j}^{(1)}\pi(u_{ij}).
\end{align*}
The same equation proves the converse assertion and the last assertion
follows from Laca's Proposition 2.2 in \cite{mL93}.\end{proof}

We conclude with a new look at Rochberg's \cite[Theorem 1]{rR73} and related
work of McDonald \cite{jMcD03}. Because of the complex conjugates
that appear in the formula for the inner product on $\LiL$, it is
somewhat surprising that $\< m_{i},f\>\in H^{\infty}(\T)$
whenever $f\in H^{\infty}(\T)$ and $m_{i}$ comes from an orthonormal
basis for $H^{2}(\T)\ominus\pi(b)H^{2}(\T)$.
\begin{theorem} \label{Thm:Rochberg's Theorem 1}Let $\{v_{i}\}_{i=1}^{N}$
be an orthonormal basis for $\mathcal{D} = H^{2}(\T)\ominus\pi(b)H^{2}(\T)$ and
let $m_{i}=v_{i}J^{-\frac{1}{2}}$, so that $\{m_{i}\}_{i=1}^{N}$
is an orthonormal basis for $\LiL$ by Corollary \ref{Cor: module basis}.
Then a function $f\in L^{\infty}(\T)$ lies in $H^{\infty}(\T)$ if
and only if $\< m_{i},f\>$ lies in $H^{\infty}(\T)$ for
all $i$. Further, $f$ lies in the disc algebra $A(\mathbb{D})$
if and only if $\< m_{i},f\>$ lies in $A(\mathbb{D})$
for all $i$. \end{theorem}

\begin{proof} 
By Remark~\ref{Canonical Basis} we know the functions $m_i$ are in the disc algebra.  It is thus immediate from
\begin{equation}\label{expansion}
f = \sum_{i=1}^n m_i \cdot \<m_i, f\>, \qquad f \in L^{\infty}(\T),
\end{equation}
and the fact that $\beta$ preserves both $H^{\infty}(\T)$ and $A(\D)$ that if the coefficients $\<m_i, f\>$ all lie in $H^{\infty}(\T)$ or $A(\D)$ then $f$ will also.  

Conversely, fix $f \in H^{\infty}(\T)$ and any $v \in \mathcal{D}$.  We must show that $\<v J^{-1/2}, f\>$ is in $H^{\infty}(\T)$.  Note that $\<v J^{-1/2}, f\>$ is in $L^{\infty}$ so it suffices to show that this function is in $H^2(\T)$.  To this end, fix a positive integer $k$, and compute
\begin{align*}
(\<v J^{-1/2}, f\>, e_{-k}) & = (\mathcal{L}(\overline{v J^{-1/2}} f), e_{-k}) \\
& = (\Gamma_b^* (J_0 \overline{v J^{-1/2}} f), e_{-k}) & \text{by Theorem~\ref{Theorem: transfer}}\\
& = (J_0 \overline{v J^{-1/2}} f, b^{-k}) \\
& = (J^{1/2} \overline{v} f, b^{-k}) & \text{as $J_0 = |J| = J^{1/2} \overline{J^{1/2}}$}\\
& = (J^{1/2} f b^k, v).
\end{align*}
Since $J^{1/2} f \in H^{\infty}$ and $k > 0$ the function $J^{1/2} f b^k$ is in $\pi(b) H^2(\T)$, so as $v \in \mathcal{D}$ we conclude $(J^{1/2} f b^k, v) = 0$.  As $k > 0$ was arbitrary, $\<v J^{-1/2}, f\>$ is in $H^2(\T)$, as desired.  If $f$ is further assumed to be in $A(\D)$, as $\Lb$ maps $C(\T)$ into itself we conclude $\<v J^{-1/2}, f\> \in C(\T) \cap H^2(\T) = A(\D)$.
\end{proof}
In our notation, Rochberg's Theorem 1 in \cite{rR73} asserts that if $\{v_i\}_{i=1}^N$ is the \emph{canonical} orthonormal basis for $\mathcal{\D}$, then for any $f \in A(\D)$, there are uniquely determined $f_{1},f_{2},\cdots,f_{N}\in A(\mathbb{D})$ satisfying
\begin{equation}
f(z)=\sum_{i=1}^{N}v_{i}(z) \beta(f_i)(z),\label{eq:Rogberg's Eq 5A} \qquad z \in \overline{\D},
\end{equation}
and that moreover for each $1 \leq i \leq N$ the linear map $f\to f_{i}$ thus determined on $A(\D)$ is continuous in the norm of $A(\D)$.  

We recover this theorem by applying Theorem~\ref{Thm:Rochberg's Theorem 1} to the canonical basis $\{v_i\}_{i=1}^N$ and the function $J^{-1/2} f \in A(\D)$: it asserts that \eqref{eq:Rogberg's Eq 5A} holds with the functions $f_i = \<m_i, J^{-1/2} f\> \in A(\D)$.  The norm continuity of the $f_i$ in $f$ is immediate from this formula.  Our Theorem~\ref{Thm:Rochberg's Theorem 1} provides a slightly stronger uniqueness statement: if $f \in A(\D)$, assuming only that the $f_i$ are in $L^{\infty}(\T)$, multiplying both sides of \eqref{eq:Rogberg's Eq 5A} by $J^{-1/2}$, applying $\<m_j, -\>$, and using the fact that $\{m_i\}_{i=1}^N$ is an orthonormal basis for $\LiL$, one finds that $f_j$ must be given by the formula above.

Rochberg \cite{rR73} and McDonald \cite{jMcD03} establish more information about the $f_{i}$ using the special
structure of the canonical orthonormal basis of $\mathcal{\D}$.  Our analysis does not seem to contribute anything new to their refinements.  On the other hand, our results are explicitly independent of the choice of basis and connect to the structure of the Hilbert module $\LiL$. 

\begin{remark}\label{remark: Concluding Remark}The reader may have
noticed that if $m\in L^{\infty}(\T)$ and if $T=\pi(m)C_{b}$, then
from the calculations in Theorem \ref{Thm: Intertwine}, the norm
of $T$ is the norm of $m$ calculated in $\LiL$. This is not an
accident. The Hilbert module $\LiL$ becomes a \emph{left }module
over $L^{\infty}(\T)$ via the formula $a\cdot\xi:=a\xi$, $a\in L^{\infty}(\T)$,
$\xi\in\LiL$. This makes $\LiL$ what is known as a $C^{*}$\emph{-correspondence}
or \emph{Hilbert bimodule} over $L^{\infty}(\T)$. Further, if $\psi:\LiL\to B(L^{2}(\T)$
is defined by the formula\[
\psi(m)=\pi(m)C_{b},\qquad m\in\LiL,\]
then the pair $(\pi,\psi)$ turns out to be what is known as a \emph{Cuntz-Pimsner
covariant representation} of the pair $(L^{\infty},\LiL)$. This means,
in particular, that $\psi(m)^{*}\psi(m)=\pi(\< m,m\>)$,
as we noted in Theorem \ref{Thm: Intertwine}.
Further, the pair $(\pi,\psi)$ extends to a $C^{*}$-representation
of the so-called \emph{Cuntz-Pimsner algebra} of $\LiL$, $\mathcal{O}(\LiL)$.
We have not made use of this here, but it strikes us as worthy
of further investigation. See \cite{FMR03} for further information
about Cuntz-Pimsner algebras and their representations.\end{remark}

\end{document}